  \documentclass{article}

\usepackage{amsfonts}
\usepackage{amsmath}
\usepackage{amssymb}
\usepackage{color}
\usepackage{indentfirst}
\usepackage{mathrsfs}
\usepackage{manfnt}

\newcommand{\ff}{\varphi}
\renewcommand{\ll}{\lambda}
\newcommand{\ii}{\mathrm{i}}
\newcommand{\dd}{\mathrm{d}}

\newcommand{\NN}{\mathbb{N}} 
\newcommand{\ZZ}{\mathbb{Z}} 

\newcommand{\DD}{\mathcal{D}} 
\newcommand{\GD}{\mathcal{GD}} 
\newcommand{\MM}{\mathcal{M}} 
\newcommand{\GM}{\mathcal{GM}} 
\renewcommand{\SS}{\mathcal{S}} 
\newcommand{\GS}{\mathcal{GS}} 
\newcommand{\FF}{\mathcal{F}} 
\newcommand{\GF}{\mathcal{GF}} 
\newcommand{\PP}{\mathcal{P}} 
\newcommand{\GP}{\mathcal{GP}} 
\newcommand{\YY}{\mathcal{Y}} 
\newcommand{\HH}{\mathcal{H}} 
\newcommand{\PPP}{\mathscr{P}} 

\newcommand{\mchoose}[2]{ { \left(\!\!\!\left( { #1 \atop #2 } \right)\!\!\!\right) } }
\newcommand{\Mchoose}[2]{ { \left(\!\!\left( { #1 \atop #2 } \right)\!\!\right) } }

\newtheorem{theorem}{Theorem}
\newtheorem{proposition}[theorem]{Proposition} 

\newcommand{\eproof}{ \hspace*{ \fill } $ \Box $ \vspace*{ 0.3cm } }
\newenvironment{proof}{ {\em Proof}. }{ \eproof }

\title{Enumeration of edges in some lattices of paths}
\author{Luca Ferrari\thanks{Dipartimento di Sistemi e Informatica, viale Morgagni 65, 50134 Firenze, Italy
{\tt ferrari@dsi.unifi.it}}
\and
Emanuele Munarini\thanks{Politecnico di Milano, Dipartimento di Matematica,
Piazza Leonardo da Vinci 32, 20133 Milano, Italy {\tt emanuele.munarini@polimi.it}}}
\date{} 

\begin{document}

\maketitle

\begin{abstract}
 We enumerate the edges in the Hasse diagram of several lattices
 arising in the combinatorial context of lattice paths.
 Specifically, we will consider the case of Dyck, Grand Dyck, Motzkin, Grand Motzkin,
 Schr\"oder and Grand Schr\"oder lattices.
 Finally, we give a general formula for the number of edges in an arbitrary Young lattice
 (which can be interpreted in a natural way as a lattice of paths).
\end{abstract}



\bigskip


\section{Introduction}

Fixed a Cartesian coordinate system in the discrete plane $\; \ZZ\times\ZZ \,$,
for any class of lattice paths starting from the origin and ending at a same point on the $x$-axis,
we can define a partial order by declaring that a path $\; \gamma_1 \;$
is less than or equal to another path $\; \gamma_2 \;$
whenever $\; \gamma_1 \;$ lies weakly below $\; \gamma_2 \,$.
In some cases, the resulting poset has a structure of distributive lattice \cite{FerrariPinzani}.
This is true, for instance, for the more common classes of paths:
\emph{Dyck} and \emph{Grand Dyck paths}, \emph{Motzkin} and \emph{Grand Motzkin paths},
\emph{Schr\"oder} and \emph{Grand Schr\"oder paths},
in correspondence of which we have \emph{Dyck} and \emph{Grand Dyck lattices},
\emph{Motzkin} and \emph{Grand Motzkin lattices},
\emph{Schr\"oder} and \emph{Grand Schr\"oder lattices}.
In all these cases, considering the paths according to their length or semi-length,
we have a sequence of distributive lattices.
Another important class is given by the (finite) \emph{Young lattices},
which can be interpreted in a very natural way as lattices of paths.
In particular, the Dyck lattices are isomorphic to the Young lattices associated with a staircase shape,
and the Grand Dyck lattices are isomorphic to the Young lattices associated with a rectangular shape.

In this paper, we enumerate the edges of the Hasse diagram for all distributive lattices recalled above.
For the classical paths, using the method of path decomposition,
we obtain the generating series with respect to certain parameters
and then (applying Proposition \ref{prop-edge-delta-nabla}) we obtain the edge series.
For the Young lattices, we obtain a general formula for the number of edges in the Hasse diagram of $\; \YY_\ll \;$
which is valid for an arbitrary partition $\; \ll \,$.

We also introduce the \emph{Hasse index} of a poset as the density index of the associated Hasse diagram
(the quotient between the number of edges and the number of vertices)
and we prove that the Hasse index of all sequences of lattices of classical paths
is always related (equal, asymptotical equivalent, or asymptotically quasi-equivalent)
to the Hasse index of Boolean lattices.

The enumeration of the edges in the Hasse diagram of a poset $\; P \;$
can be considered as a specialization of the more general case of enumerating all saturated chains in $\; P \,$,
which is clearly much more complicated.
In \cite{FerrariMunariniChains},
we obtain a general formula for counting saturated chains of any finite length $\; k \;$ in any Dyck lattice.

\section{Background}

\subsection{Lattice paths}\label{subsec-paths}

In this paper, we will consider the following classes of paths.
\begin{enumerate}
 \item
  The class $\; \GD \;$ of \emph{Grand Dyck paths},
  i.e. the class of all lattice paths starting from the origin, ending on the $x$-axis,
  and consisting of \emph{up steps} $\; U = (1,1) \;$ and \emph{down steps} $\; D = (1,-1) \,$.
  The class $\; \DD \;$ of \emph{Dyck paths} \cite{Deutsch,Stanley2},
  consisting of all Grand Dyck paths never going below the $x$-axis.
 \item
  The class $\; \GM \;$ of \emph{Grand Motzkin paths},
  i.e. the class of all lattice paths starting from the origin, ending on the $x$-axis,
  and consisting of \emph{up steps} $\; U = (1,1) \,$, \emph{down steps} $\; D = (1,-1) \;$
  and \emph{horizontal steps} $\; H = (1,0) \,$.
  The class $\; \MM \;$ of \emph{Motzkin path} \cite{Stanley2},
  consisting of all Grand Motzkin paths never going below the $x$-axis.
 \item
  The class $\; \GS \;$ of \emph{Grand Schr\"oder paths} (or \emph{central Delannoy paths}),
  i.e. the class of all lattice paths starting from the origin, ending on the $x$-axis,
  and consisting of \emph{up steps} $\; U = (1,1) \,$, \emph{down steps} $\; D = (1,-1) \;$
  and \emph{double horizontal steps} $\; H = (2,0) \,$.
  The class $\; \SS \;$ of \emph{Schr\"oder path},
  consisting of all Grand Schr\"oder paths never going below the $x$-axis.
\end{enumerate}

The \emph{length} of a path is the number of its steps,
and the \emph{semi-length} of a path of even length is half of the number of its steps.
In general, for a class $\; \PP \;$ containing paths of every length
(as in the case of Grand Motzkin paths and Motzkin paths)
we denote by $\; \PP_n \;$ the set of all paths in $\; \PP \;$ having length $\; n \,$.
If the class $\; \PP \;$ contains only paths of even length
(as in the case of Grand Dyck paths, Dyck paths, Grand Schr\"oder paths and Schr\"oder paths)
we denote by $\; \PP_n \;$ the set of all paths in $\; \PP \;$ having semi-length $\; n \,$.
Sometimes, we write $\; \bullet \;$ for the \emph{empty path} (consisting of zero steps).

For a class $\; \PP \;$ containing paths never going below the $x$-axis,
we define the class $\; \overline{\PP} \;$ of \emph{reflected paths}
as the set of all paths obtained by reflecting about the $x$-axis the paths in $\; \PP \,$.
In particular, we will consider the class $\; \overline{\DD} \;$ of \emph{reflected Dyck paths},
the class $\; \overline{\MM} \;$ of \emph{reflected Motzkin paths},
and the class $\; \overline{\SS} \;$ of \emph{reflected Schr\"oder paths}.
Moreover, we say that a \emph{reflected path} $\; \gamma \in \overline{\PP} \;$ is \emph{subelevated}
when $\; \gamma = D\gamma'U \,$, for any $\; \gamma' \in \overline{\PP} \;$.

A path can always be considered as a word on the alphabet given by the set of possible steps.
A factor of a path $\; \gamma \;$ (considered as a word) is a word $\; \alpha \;$
such that $\; \gamma = \gamma'\alpha\gamma'' \,$.
We write $\; \omega_\alpha(\gamma) \;$ for the number
of all occurrences of the word $\; \alpha \;$ as a factor of $\; \gamma \,$.
In particular, if $\; \alpha \;$ is a path starting and ending at the same level,
then we write $\; \omega^*_\alpha(\gamma) \;$ for the number
of all occurrences of the word $\; \alpha \;$ as a factor of $\; \gamma \;$ not on the $x$-axis.
For instance, if $\; \gamma = UUDHDHUHD \,$, then $\; \omega_H(\gamma) = 3 \;$
and $\; \omega^*_H(\gamma) = 2 \;$ (since the second horizontal step lies on the $x$-axis).

\subsection{Enumeration and asymptotics}

For simplicity, we recall some well known enumerative properties of the paths considered in Subsection \ref{subsec-paths}
that will be used in the rest of the paper.

The Grand Dyck paths are enumerated by the
\emph{central binomial coefficients} $\; { 2n \choose n } \;$ \cite[A000984]{Sloane},
and the Dyck paths are enumerated by
the \emph{Catalan numbers} $\; C_n = { 2n \choose n }\frac{1}{n+1} \;$ \cite[A000108]{Sloane}.
These numbers have generating series
$$
 B(x) = \sum_{n\geq0} { 2n \choose n } x^n = \frac{1}{\sqrt{1-4x}}\, , \qquad
 C(x) = \sum_{n\geq0} C_n x^n = \frac{1-\sqrt{1-4x}}{2x} \, .
$$

The \emph{trinomial coefficient} $\; { n;\,3 \choose k } \;$
is defined as the coefficient of $\; x^k \;$ in the expansion of $\; (1+x+x^2)^n \;$ \cite{Comtet} \cite[A027907]{Sloane}.
The Grand Motzkin paths are enumerated by
the \emph{central trinomial coefficients} $\; { n;\,3 \choose n } \;$ \cite[A002426]{Sloane},
and the Motzkin paths are enumerated by the \emph{Motzkin numbers} \cite[A001006]{Sloane}.
These numbers have generating series
$$
 T(x) = \sum_{k\geq0} { n;\,3 \choose n } x^n = \frac{1}{\sqrt{1-2x-3x^2}}\, ,
 \qquad
 M(x) = \sum_{k\geq0} M_n\, x^n = \frac{1-x-\sqrt{1-2x-3x^2}}{2x^2} \, .
$$

The Grand Schr\"oder paths are enumerated by the \emph{central Delannoy numbers} $\; d_n \;$ \cite[A001850]{Sloane},
and the Schr\"oder paths are enumerated by the \emph{large Schr\"oder numbers} $\; r_n \;$ \cite[A006318]{Sloane}.
These numbers have generating series
$$
 d(x) = \sum_{n\geq0} d_n x^n = \frac{1}{\sqrt{1-6x+x^2}}\, ,
 \qquad
 r(x) = \sum_{n\geq0} r_n x^n = \frac{1-x-\sqrt{1-6x+x^2}}{2x} \, .
$$

Given two sequences $\; a_n \;$ and $\; b_n \;$ (where $\; b_n \;$ is definitively non zero),
we recall that the notation $\; a_n \sim b_n \;$ means that $\; a_n/b_n \to 1 \;$ as $\; n\to+\infty \,$.
Moreover, we recall that the \emph{Darboux theorem} \cite[p. 252]{BergeronLabelleLeroux} says that:
given a complex number $\; \xi \ne 0 \;$ and a complex function $\; f(x) \;$ analytic at the origin,
if $\; f(x) = (1-x/\xi)^{-\alpha}\psi(x) \;$ where $\; \psi(x) \;$ is a series with radius of convergence $\; R > |\xi| \;$
and $\; \alpha \not \in \{ 0, -1, -2, \ldots \} \,$, then
$$
 [x^n]f(x) \sim \frac{ \psi( \xi ) }{ \xi^n }\, \frac{ n^{ \alpha - 1 } }{ \Gamma( \alpha ) } \, ,
$$
where $\; \Gamma(z) \;$ is Euler's Gamma function.
Using such a theorem, it is possible to obtain the following asymptotic expansions:
\begin{equation}\label{asymptotics}
 { n;\; 3 \choose n } \sim \frac{3^n}{2}\sqrt{\frac{3}{n\pi}}\, , \quad
 M_n \sim \frac{3^{n+1}}{2n}\sqrt{\frac{3}{n\pi}}\, , \quad
 d_n \sim \frac{(1+\sqrt{2})^{2n+1}}{2\sqrt{\sqrt{2}n\pi}}\, , \quad
 r_n \sim \frac{(1+\sqrt{2})^{2n+1}}{n\sqrt{2\sqrt{2}n\pi}} \, .
\end{equation}

Finally, we recall the following elementary expansions
$$
 \frac{x^r}{(1-x)^{s+1}} = \sum_{n\geq0} { n-r+s \choose s }\, x^n
 \qquad\text{and}\qquad
 \frac{1}{(1-x)^s} = \sum_{n\geq0} \Mchoose{s}{n}\, x^n
$$
where $\; \Mchoose{n}{k} = \frac{n(n+1)\cdots(n+k-1)}{k!} \;$ are the \emph{multiset coefficients}.

\subsection{Edge enumeration}

The \emph{Hasse diagram} $\; \HH(P) \;$ of a finite poset $\; P \;$ is a (directed) graph representing the poset,
where the vertices are the elements of $\; P \;$ and the adjacency relation is the cover relation.
The number of vertices is $\; |P| \,$.
We will denote with $\; \ell(P) \;$ the number of all edges in $\; \HH(P) \,$.
If $\; \Delta x \;$ is the set of all elements covering $\; x \;$
and $\; \nabla x \;$ is the set of all elements covered by $\; x \,$, then we have
$$ \ell(P) = \sum_{x\in P} | \Delta x | = \sum_{x\in P} | \nabla x | \, . $$
Moreover, if we consider the polynomials
$$
 \Delta(P;q) = \sum_{x\in P} q^{| \Delta x |}
 \qquad\text{and}\qquad
 \nabla(P;q) = \sum_{x\in P} q^{| \nabla x |} \, ,
$$
then we have at once the identities
\begin{equation}\label{id-edges-delta-nabla}
 \ell(P) = \left[\partial_q\Delta(P;q)\right]_{q=1}
 \qquad\text{and}\qquad
 \ell(P) = \left[\partial_q\nabla(P;q)\right]_{q=1} \, ,
\end{equation}
where $\; \partial_q \;$ denotes the partial derivative with respect to $\; q \,$.
The \emph{edge generating series} of a sequence of posets $\; \PPP = \{ P_0, P_1, P_2, \ldots \} \;$
is the ordinary generating series $\; \ell_\PPP(x) \;$ of the numbers $\; \ell(P_n) \,$.
Similarly, the \emph{$\Delta$-series} and \emph{$\nabla$-series} associated with the sequence $\; \PPP \;$
are the generating series $\; \Delta_\PPP(q;x) \;$ and $\; \nabla_\PPP(q;x) \;$
of the polynomials $\; \Delta(P_n;q) \;$ and $\; \nabla(P_n;q) \,$, respectively.
From identities (\ref{id-edges-delta-nabla}), we have at once
\begin{proposition}\label{prop-edge-delta-nabla}
 The edge generating series for the sequence of posets $\; \PPP = \{ P_0, P_1, P_2, \ldots \} \;$
 can be obtained from the associated $\Delta$-series and $\nabla$-series as follows
 $$
  \ell_\PPP(x) = \left[\partial_q \Delta_\PP(q;x)\right]_{q=1}
  \qquad\text{and}\qquad
  \ell_\PPP(x) = \left[\partial_q \nabla_\PP(q;x)\right]_{q=1} \, .
 $$
\end{proposition}

The \emph{density index} $\; i(G) \;$ of a graph $\; G \;$ is the quotient
between the number of edges and the number of vertices, i.e. $\; i(G) = |E(G)|/|V(G)| \,$.
This index has been considered in the study of topologies for the interconnection of parallel multicomputers,
especially in the attempt of finding alternative topologies to the classical one given by the Boolean cube
\cite{Hsu,MunariniPerelliZagaglia}, and in many other circumstances
(as, for instance, in \cite{ClimerTempletonZhang} or in \cite{MolloyReed}).
Here, we define the \emph{Hasse index} $\; i(P) \;$ of a poset $\; P \;$
as the density index of its Hasse diagram, i.e. $\; i(P) = \ell(P)/|P| \,$.
For instance, the Hasse index of a \emph{Boolean lattice} $\; B_n \;$ is $\; i(B_n) = \ell(B_n)/|B_n| = n/2 \,$,
since $\; |B_n| = 2^n \;$ and $\; \ell(B_n) = n 2^{n-1} \,$.
We say that the Hasse index of a sequence of posets $\; \PPP = \{ P_0, P_1, P_2, \ldots \} \;$
is \emph{Boolean} when $\; i(P_n) = n/2 \,$,
is \emph{asymptotically Boolean} when $\; i(P_n) \sim n/2 \;$ as $\; n \to +\infty \,$,
and is \emph{asymptotically quasi Boolean} when
there exists a small non-negative constant $\; c \;$ such that $\; i(P_n) \sim (1/2\pm c)\, n \;$ as $\; n \to +\infty \,$.
Here, we can assume $\; c \leq 1/10 \,$.

Let $\; \GP \;$ be a class of lattice paths (the \emph{Grand paths})
starting from the origin, ending on the $x$-axis
consisting of steps of some kind (and respecting possible restrictions).
Then let $\; \PP \;$ be the class of all paths in $\; \GP \;$ never going below the $x$-axis.
We say that the class $\; \PP \;$ is \emph{Hasse-tamed}
if the Hasse index of the associated posets of paths is asymptotically equivalent
to the Hasse index of the associated posets of Grand paths,
i.e. $\; i(\GP_n) \sim i(\PP_n) \;$ as $\; n \to +\infty \,$.
In all main examples we will consider, the property of being Hasse-tamed is true.
However, there are also classes of paths without such a property,
as in the case of the Fibonacci paths considered in Section \ref{sec-Fibo}.

In the rest of the paper, given a class $\; \PP \;$ of paths,
we write $\; \ell_\PP(x) \;$ for the edge generating series $\; \ell_\PPP(x) \;$
associated with the sequence $\; \PPP \;$ of posets generated by all paths in $\; \PP \,$.

For convenience, we report in Table \ref{Table-edges}
the first few values of the number of edges for the various lattices we will consider in the paper.
Moreover, we observe that they appear in \cite{Sloane} as follows:
$\; \ell(\FF_n) \;$ form sequence A001629,
$\; \ell(\GF_n) \;$ form sequence A095977,
$\; \ell(\DD_n) \;$ form sequence A002054,
$\; \ell(\GD_n) \;$ form sequence A002457,
$\; \ell(\MM_n) \;$ form sequence A025567,
$\; \ell(\GM_n)/2 \;$ form sequence A132894,
$\; \ell(\GS_n)/2 \;$ form sequence A108666.
\begin{table}[h]
$$
 \begin{array}{|c|ccccccccccc|}
  \hline
   n & 0 & 1 & 2 & 3 & 4 & 5 & 6 & 7 & 8 & 9 & 10 \\ 
  \hline
   \ell(\FF_n) & 0 & 0 & 1 & 2 & 5 & 10 & 20 & 38 & 71 & 130 & 235 \\
   \ell(\GF_n) & 0 & 0 & 2 & 4 & 14 & 32 & 82 & 188 & 438 & 984 & 2202 \\
   \ell(\DD_n) & 0 & 0 & 1 & 5 & 21 & 84 & 330 & 1287 & 5005 & 19448 & 75582 \\ 
   \ell(\GD_n) & 0 & 1 & 6 & 30 & 140 & 630 & 2772 & 12012 & 51480 & 218790 & 923780 \\ 
   \ell(\MM_n) & 0 & 0 & 1 & 4 & 13 & 40 & 120 & 356 & 1050 & 3088 & 9069 \\ 
   \ell(\GM_n) & 0 & 0 & 2 & 8 & 30 & 104 & 350 & 1152 & 3738 & 12000 & 38214 \\
   \ell(\SS_n) & 0 & 1 & 6 & 34 & 190 & 1058 & 5894 & 32898 & 184062 & 1032322 & 5803270 \\ 
   \ell(\GS_n) & 0 & 2 & 16 & 114 & 768 & 5010 & 32016 & 201698 & 1257472 & 7777314 & 47800080 \\
  \hline
 \end{array}
$$
\caption{Number of edges in some lattices of paths.}
\label{Table-edges}
\end{table}

\section{Dyck and Grand Dyck lattices}

\begin{proposition}\label{prop-delta-nabla-Dyck}
 For any Dyck path $\; \gamma \,$, we have $\; | \Delta\gamma | = \omega_{DU}(\gamma) \;$
 and $\; | \nabla\gamma | = \omega^*_{UD}(\gamma) \,$.
 Similarly, for any Grand Dyck path $\; \gamma \,$, we have $\;| \Delta\gamma | = \omega_{DU}(\gamma) \;$
 and $\; | \nabla\gamma | = \omega_{UD}(\gamma) \,$.
\end{proposition}
\begin{proof}
 In a Dyck lattice, a path $\; \gamma \;$ is covered by all paths that can be obtained from $\; \gamma \;$
 by replacing a valley $\; DU \;$ with a peak $\; UD \,$,
 and covers all paths that can be obtained from $\; \gamma \;$
 by replacing a peak $\; UD \;$ (not at level $\; 0 \,$) with a valley $\; DU \,$.
 In a Grand Dyck lattice, the situation is similar.
\end{proof}

\begin{proposition}
 The generating series for the class of Dyck paths
 with respect to semi-length (marked by $\; x \,$) and valleys (marked by $\; q \,$) is
 \begin{equation}\label{series-Dyck-qx}
  f(q;x) = \frac{1-(1-q)x-\sqrt{1-2(1+q)x+(1-q)^2x^2}}{2qx} \, .
 \end{equation}
\end{proposition}
\begin{proof}
 Any non-empty Dyck path $\; \gamma \;$ decomposes uniquely
 as $\; \gamma = U\gamma' D \;$ (with $\; \gamma' \in \DD \,$)
 or as $\; \gamma = U\gamma' D \gamma'' \;$ (with $\; \gamma', \gamma'' \in \DD \,$, $\; \gamma'' \ne \bullet \,$).
 Hence, we have the identity $\; f(q;x) = 1 + x f(q;x) + q x f(q;x) ( f(q;x) - 1 ) \;$
 whose solution is series (\ref{series-Dyck-qx}).
\end{proof}

Notice that series (\ref{series-Dyck-qx}) is essentially
the generating series of \emph{Narayana numbers} \cite[A001263]{Sloane}
and that this statistic is well known (see, for instance, \cite{Deutsch}).

\begin{theorem}\label{thm-Dyck}
 The edge generating series for Dyck lattices is
 \begin{equation}\label{series-Dyck-edges}
  \ell_\DD(x) = \frac{1-3x-(1-x)\sqrt{1-4x}}{2x\sqrt{1-4x}} \, .
 \end{equation}
 Moreover, for every $\; n \in \NN \,$, $\; n \geq 2 \,$, the number of edges in $\; \DD_n \;$ is
 \begin{equation}\label{id-Dyck-edges}
  \ell(\DD_n) = \frac{1}{2} { 2n \choose n } \frac{n-1}{n+1} = { 2n-1 \choose n-2 } \, .
 \end{equation}
 In particular, the Hasse index of Dyck lattices is asymptotically Boolean.
\end{theorem}
\begin{proof}
 By Proposition \ref{prop-delta-nabla-Dyck}, the $\Delta$-series for Dyck lattices is series (\ref{series-Dyck-qx}).
 So, by applying Proposition \ref{prop-edge-delta-nabla}, we can obtain series (\ref{series-Dyck-edges}).
 Moreover, since $\; \ell_\DD(x) = (1+B(x))/2-C(x) \,$, we obtain identity (\ref{id-Dyck-edges}).
 Finally, since $\; |\DD_n| = C_n \,$,
 we have $\; i(\DD_n) = \ell(\DD_n)/|\DD_n| = (n-1)/2 \;$ for every $\; n \geq 1 \,$, and $\; i(\DD_n) \sim n/2 \,$.
\end{proof}

\begin{proposition}
 The generating series for the class of Grand Dyck paths
 with respect to semi-length (marked by $\; x $) and valleys (marked by $\; q $) is
 \begin{equation}\label{series-Grand-Dyck-qx}
  F(q;x) = \frac{1}{\sqrt{1-2(1+q)x+(1-q)^2x^2}} \, .
 \end{equation}
\end{proposition}
\begin{proof}
 Let $\; \overline{\DD} \;$ be the class of reflected Dyck paths
 (i.e. Grand Dyck paths never going above the $x$-axis),
 and let $\; G(q;x) \;$ be the corresponding generating series.
 Since every non empty path $\; \gamma \in \overline{\DD} \;$ uniquely decomposes as
 $\; \gamma = DU \gamma' \;$ (with $\; \gamma' \in \overline{\DD} \,$)
 or as $\; \gamma = D \gamma' U \gamma'' \;$
 (with $\; \gamma', \gamma'' \in \overline{\DD} \,$, $\; \gamma' \ne \bullet \,$),
 we have the identity $\; G(q;x) = 1 + q x G(q;x) + x ( G(q;x) - 1 ) G(q;x) \,$,
 whose unique solution is
 $$ G(q;x) = \frac{1+(1-q)x-\sqrt{1-2(1+q)x+(1-q)^2x^2}}{2x} \, . $$

 Now, let $\; U(q;x) \;$ be the generating series for the class of Grand Dyck paths starting with an up step
 and let $\; D(q;x) \;$ be the generating series for the class of Grand Dyck paths starting with a down step.
 Since every path $\; \gamma \in \GD \;$ uniquely decomposes as product of paths of the form
 $\; U \gamma D \;$ (with $\; \gamma \in \DD \,$) and $\; D \gamma U \;$ (with $\; \gamma' \in \overline{\DD} \,$),
 we have the linear system
 $$
  \begin{cases}
   F(q;x) = 1 + U(q;x) + D(q;x) \\
   U(q;x) = x F(q;x) ( 1 + q U(q;x) + D(q;x) ) \\
   D(q;x) = q x F(q;x) + x ( G(q;x) - 1 ) F(q;x)
  \end{cases}
 $$
 from which it is straightforward to obtain identity (\ref{series-Grand-Dyck-qx}).
\end{proof}

\begin{theorem}\label{thm-GrandDyck}
 The edge generating series for Grand Dyck lattices is
 \begin{equation}\label{series-edges-Grand-Dyck}
  \ell_\GD(x) = \frac{x}{(1-4x)^{3/2}} \, .
 \end{equation}
 Moreover, the number of edges in $\; \GD_n \;$ is
 \begin{equation}\label{id-edges-Grand-Dyck}
  \ell(\GD_n) = { 2n \choose n } \frac{n}{2} \, .
 \end{equation}
 In particular, the Hasse index of a Grand Dyck lattice is Boolean.
\end{theorem}
\begin{proof}
 Proposition \ref{prop-delta-nabla-Dyck} implies that (\ref{series-Grand-Dyck-qx})
 is the $\Delta$-series for Grand Dyck lattices.
 So, by applying Proposition \ref{prop-edge-delta-nabla}, we obtain series (\ref{series-edges-Grand-Dyck}).
 Then, by expanding this series, we have at once identity (\ref{id-edges-Grand-Dyck}).
 Finally, since $\; |\GD_n| = { 2n \choose n } \,$, we have $\; i(\GD_n) = \ell(\GD_n)/|\GD_n| = n/2 \,$.
\end{proof}

Theorems \ref{thm-Dyck} and \ref{thm-GrandDyck} immediately imply
\begin{proposition}
 The class of Dyck lattices is Hasse-tamed: $\; i(\DD_n) \sim i(\GD_n) \sim n/2 \,$.
\end{proposition}

\section{Motzkin and Grand Motzkin lattices}

\begin{proposition}\label{prop-delta-nabla-Motzkin}
 For any Motzkin path $\; \gamma \,$, we have
 $\; | \Delta\gamma | = \omega_{HU}(\gamma) + \omega_{DH}(\gamma) + \omega_{DU}(\gamma) + \omega_{HH}(\gamma) \;$
 and $\; | \nabla\gamma | = \omega_{UH}(\gamma) + \omega_{HD}(\gamma) + \omega_{UD}(\gamma) + \omega^*_{HH}(\gamma) \,$.
 Similarly, for any Grand Motzkin path $\; \gamma \,$, we have
 $\; | \Delta\gamma | = \omega_{HU}(\gamma) + \omega_{DH}(\gamma) + \omega_{DU}(\gamma) + \omega_{HH}(\gamma) \;$
 and $\; | \nabla\gamma | = \omega_{UH}(\gamma) + \omega_{HD}(\gamma) + \omega_{UD}(\gamma) + \omega_{HH}(\gamma) \,$.
\end{proposition}
\begin{proof}
 In a Motzkin lattice, a path $\; \gamma \;$ is covered by all paths that can be obtained from $\; \gamma \;$ by replacing
 i) a factor $\; HU \;$ with a factor $\; UH \,$, or
 ii) a factor $\; DH \;$ with a factor $\; HD \,$, or
 iii) a valley $\; DU \;$ with a double horizonal step $\; HH \,$, or
 iv) a double horizonal step $\; HH \;$ with a peak $\; UD \,$.
 The number of paths covered by $\; \gamma \;$ can be obtained in a similar way.
 In a Grand Motzkin lattice we have a similar situation.
\end{proof}

\begin{proposition}
 The generating series for the class of Motzkin paths with respect to length (marked by $\; x $)
 and to factors $\; HU \,$, $\; DH \,$, $\; DU \;$ and $\; HH \;$ (marked by $\; q $) is
 \begin{equation}\label{series-Motzkin-qx}
  f(q;x) = \frac{1-qx-(1-q)x^2-\sqrt{(1+x)(1-(1+2q)x-(1-q^2)x^2+(1-q)^2x^3)}}{2qx^2} \, .
 \end{equation}
\end{proposition}
\begin{proof}
 Let $\; h(q;x) \;$ be the generating series for the class of Motzkin paths starting with an horizontal step
 and let $\; u(q;x) \;$ be the generating series for the class of Motzkin paths starting with an up step.
 Since any non-empty Motzkin path $\; \gamma \;$ decomposes uniquely
 as $\; \gamma = H\gamma' \;$ (with $\; \gamma' \in \MM \,$),
 or as $\; \gamma = U\gamma'D\gamma'' \;$ (with $\; \gamma', \gamma'' \in \MM \,$),
 it is straightforward to obtain the linear system
 $$
  \begin{cases}
   f(q;x) = 1 + h(q;x) + u(q;x) \\
   h(q;x) = x ( 1 + q h(q;x) + q u(q;x) ) \\
   u(q;x) = x^2 f(q;x) ( 1 + q h(q;x) + q u(q;x) )
  \end{cases}
 $$
 and consequently, solving such a system, to obtain identity (\ref{series-Motzkin-qx}).
\end{proof}

The numbers generated by series (\ref{series-Motzkin-qx}) are essentially sequence A110470 in \cite{Sloane}.

\begin{theorem}\label{thm-Motzkin}
 The edge generating series for Motzkin lattices is
 \begin{equation}\label{series-Motzkin-edges}
  \ell_\MM(x) = \frac{(1+x)(1-2x-x^2-(1-x)\sqrt{1-2x-3x^2})}{2x^2\sqrt{1-2x-3x^2}} \, .
 \end{equation}
 Moreover, the number of edges in $\; \MM_n \;$ can be expressed in one of the following ways
 \begin{align}
  & \ell(\MM_n) = { n;\, 3 \choose n } - M_n + { n-1;\, 3 \choose n - 1 } - M_{n-1} \qquad (n\geq 1) \label{id-Motzkin-edges-01} \\
  & \ell(\MM_n) = { n;\, 3 \choose n - 2 } + { n - 1;\, 3 \choose n - 3 } \qquad (n\geq 3) \label{id-Motzkin-edges-02} \\
  & \ell(\MM_n) = \frac{2}{n} \sum_{k=0}^{\lfloor n/2 \rfloor} { n \choose k } { n-k \choose k }\frac{k(n-k)}{k+1}
  \qquad(n\geq 1) \label{id-Motzkin-edges-03}\, .
 \end{align}
 In particular, we have the asymptotic expansions
 \begin{equation}\label{asym-Motzkin-edges}
  \ell(\MM_n) \sim \frac{2\cdot 3^n}{\sqrt{3n\pi}}
  \qquad\text{and}\qquad
  i(\MM_n) = \frac{\ell(\MM_n)}{|\MM_n|} \sim \frac{4}{9}\;n
 \end{equation}
 and the Hasse index of the Motzkin lattices is asymptotically quasi Boolean.
\end{theorem}
\begin{proof}
 By Proposition \ref{prop-delta-nabla-Motzkin},
 the $\Delta$-series for Motzkin lattices is series (\ref{series-Motzkin-qx}).
 So, by Proposition \ref{prop-edge-delta-nabla}, we obtain series (\ref{series-Motzkin-edges}).
 It is easy to see that $\; \ell_\MM(x) = (1+x)(T(x)-M(x)) \,$,
 and consequently to obtain the first identity (\ref{id-Motzkin-edges-01}).
 Now, by Cauchy integral formula, we have
 $$ \ell(\MM_n) = [x^n]\ell_\MM(x) = \frac{1}{2\pi\ii} \oint \ell_\MM(z)\; \frac{\dd z}{z^{n+1}}\, . $$
 With the substitution $\; z = \frac{w}{1+w+w^2} \,$, we have $\; \dd z = \frac{1-w^2}{(1+w+w^2)^2}\,\dd w \;$ and
 \begin{eqnarray*}
  \lefteqn{\ell(\MM_n) = \frac{1}{2\pi\ii} \oint (1+w)^2(1+w+w^2)^{n-1}\; \frac{\dd w}{w^{n-1}}=} && \\
  && = [x^{n-2}] (1+x)^2(1+x+x^2)^{n-1}
     = [x^{n-2}] (1+x+x^2)^n + [x^{n-3}](1+x+x^2)^{n-1}
 \end{eqnarray*}
 from which we have identity (\ref{id-Motzkin-edges-02}).

 From the identity
 $$
  \ell_\MM(x) = (1+x)(T(x)-M(x))
  = \frac{1+x}{1-x}\; B\left(\frac{x^2}{(1-x)^2}\right) - \frac{1+x}{1-x}\; C\left(\frac{x^2}{(1-x)^2}\right)
 $$
 we have the expansion
 \begin{eqnarray*}
  \ell_\MM(x)
  & = & (1+x) \sum_{k\geq0} { 2k \choose k } \frac{x^{2k}}{(1-x)^{2k+1}}
  - (1+x) \sum_{k\geq0} { 2k \choose k } \frac{1}{k+1}\; \frac{x^{2k}}{(1-x)^{2k+1}} \\
  & = & (1+x) \sum_{n\geq0} \left[\sum_{k\geq0} { n \choose 2k } { 2k \choose k } \frac{k}{k+1}\right] x^n
 \end{eqnarray*}
 from which it is straightforward to obtain identity (\ref{id-Motzkin-edges-03}).

 Finally, using identity (\ref{id-Motzkin-edges-01})
 and the asymptotic expansions reported in (\ref{asymptotics})
 for the central trinomial coefficients and for the Motzkin numbers,
 we can obtain the first asymptotic equivalence in (\ref{asym-Motzkin-edges}).
 Then, using once again the asymptotic expansion for the Motzkin numbers in (\ref{asymptotics}),
 we also obtain the second asymptotic equivalence in (\ref{asym-Motzkin-edges}).
 Since $\; 4/9 \simeq 0.44 \,$, the Hasse index is asymptotically quasi boolean.
\end{proof}

\begin{proposition}
 The generating series for the class of Grand Motzkin paths with respect to semi-length (marked by $\; x $)
 and to factors $\; HU \,$, $\; DH \,$, $\; DU \;$ and $\; HH \;$ (marked by $\; q $) is
 \begin{equation}\label{series-Grand-Motzkin-qx}
  F(q;x) = \frac{1+(1-q)x}{\sqrt{(1+x)(1-(1+2q)x-(1-q^2)x^2+(1-q)^2x^3)}} \, .
 \end{equation}
\end{proposition}
\begin{proof}
 Let $\; \overline{\MM} \;$ be the class of reflected Motzkin paths
 (i.e. Grand Motzkin paths never going above the $x$-axis).
 Then, let $\; X(q;x) \;$ be the generating series for the class of Grand Motzkin paths
 starting with a step $\; X \in \{ H, U, D \} \,$.
 Any non-empty Grand Motzkin path $\; \gamma \;$ decomposes uniquely
 as $\; \gamma = H\gamma' \;$ (with $\; \gamma'\in\GM \,$),
 as $\; \gamma = U\gamma'D\gamma'' \;$ (with $\; \gamma'\in\MM \;$ and $\; \gamma''\in\GM \,$),
 or as $\; \gamma = D\gamma'U\gamma'' \;$ (with $\; \gamma'\in\overline{\MM} \;$ and $\; \gamma''\in\GM \,$).
 From this decomposition, we can obtain the linear system
 $$
  \begin{cases}
   f(q;x) = 1 + H(q;x) + U(q;x) + D(q;x) \\
   H(q;x) = x ( 1 + q H(q;x) + q U(q;x) + D(q;x) ) \\
   U(q;x) = x^2 f(q;x) ( 1 + q H(q;x) + q U(q;x) + D(q;x) ) \\
   D(q;x) = \overline{D}(q;x) f(q;x) \, ,
  \end{cases}
 $$
 where $\; f(q;x) \;$ is series (\ref{series-Motzkin-qx}) and
 $\; \overline{D}(q;x) \;$ is the generating series for the class of subelevated reflected Motzkin paths.

 Now, let $\; \overline{f}(q;x) \,$, $\; \overline{h}(q;x) \,$, and $\; \overline{d}(q;x) \;$
 be the generating series for the classes of reflected Motzkin paths with no restriction,
 starting with a horizonal step and starting with a down step, respectively.
 Moreover, let $\; \overline{H}(q;x) \;$ be the generating series for the class of reflected Motzkin paths
 starting and ending with a horizontal step.
 Any non-empty reflected Motzkin path $\; \gamma \;$ decomposes uniquely
 as $\; \gamma = H\gamma' \;$ (with $\; \gamma'\in\overline{\MM} \,$),
 or as $\; \gamma = D\gamma'U\gamma'' \;$ (with $\; \gamma',\gamma''\in\overline{\MM} \,$).
 From this decomposition it is possible to obtain the linear system
 $$
  \begin{cases}
   \overline{f}(q;x) = 1 + \overline{h}(q;x) + \overline{d}(q;x) \\
   \overline{h}(q;x) = x ( 1 + q \overline{h}(q;x) + \overline{d}(q;x) ) \\
   \overline{d}(q;x) = \overline{D}(q;x) \overline{f}(q;x) \\
   \overline{D}(q;x) = x^2 ( q + q^2 x + 2 q \overline{h}(q;x) \overline{D}(q;x) + q^2 \overline{H}(q;x)
    + \overline{D}(q;x) + \overline{D}(q;x)^2 \overline{f}(q;x) ) \\
   \overline{H}(q;x) = x^2 ( q + q^2 x + 2 q \overline{h}(q;x) \overline{D}(q;x) + q^2 \overline{H}(q;x)
    + \overline{D}(q;x) + \overline{D}(q;x)^2 \overline{f}(q;x) ) \, . 
  \end{cases}
 $$

 By solving both these systems, it is straightforward to obtain series (\ref{series-Grand-Motzkin-qx}).
\end{proof}

\begin{theorem}\label{thm-GrandMotzkin}
 The edge generating series for Grand Motzkin lattices is
 \begin{equation}\label{series-Grand-Motzkin-edges}
  \ell_\GM(x) = \frac{2x^2}{(1-3x)\sqrt{1-2x-3x^2}} \, .
 \end{equation}
 Moreover, we have the identities
 \begin{align}
  & \ell(\GM_{n+2}) = 2 \sum_{k=0}^n { k;\, 3 \choose k } 3^{n-k}\label{id-Grand-Motzkin-edges-01} \\
  & \ell(\GM_{n+2}) = \frac{2}{4^n} \sum_{k=0}^n { 2k \choose k }{ 2n-2k \choose n-k } (2k+1) 3^k (-1)^{n-k} \label{id-Grand-Motzkin-edges-02} \\
  & \ell(\GM_{n+2}) = 2 \sum_{k=0}^n { n+1 \choose k+1 } { 2k \choose k } (-1)^k 3^{n-k} \label{id-Grand-Motzkin-edges-03}
 \end{align}
 and the asymptotic equivalences
 \begin{equation}\label{asym-Grand-Motzkin-edges}
  \ell(\GM_n) \sim 2\cdot 3^{n-2}\;\sqrt{\frac{3n}{\pi}}
  \qquad\text{and}\qquad
  i(\GM_n) \sim \frac{4}{9}\; n \, .
 \end{equation}
 In particular, the Hasse index of Grand Motzkin lattices is asymptotically quasi Boolean.
\end{theorem}
\begin{proof}
 Proposition \ref{prop-delta-nabla-Motzkin} implies that (\ref{series-Grand-Motzkin-qx})
 is the $\Delta$-series for Grand Motzkin lattices.
 So, by applying Proposition \ref{prop-edge-delta-nabla}, we obtain series (\ref{series-Grand-Motzkin-edges}).
 Since $\; \frac{\ell_\GM(x)}{x^2} = \frac{2}{1-3x}\;T(x) \,$, we have at once identity (\ref{id-Grand-Motzkin-edges-01}).
 Moreover, we also have
 $$
  \frac{1}{2}\,\frac{\ell_\GM(x)}{x^2} = \frac{1}{(1-3x)^{3/2}}\frac{1}{\sqrt{1+x}}
  = \sum_{n\geq0} { 2n \choose n } \frac{3^n}{4^n}\,(2n+1)\; x^n\;\cdot\;
    \sum_{n\geq0} { 2n \choose n } \frac{(-1)^n}{4^n}\; x^n
 $$
 from which we obtain identity (\ref{id-Grand-Motzkin-edges-02}).
 Finally, we have
 \begin{eqnarray*}
  \lefteqn{\frac{1}{2}\,\frac{\ell_\GM(x)}{x^2} = \frac{1}{(1-3x)\sqrt{(1-3x)^2+4x(1-3x)}}
  = \frac{1}{(1-3x)^2} \frac{1}{\sqrt{1+\frac{4x}{1-3x}}} =} && \\
  && = \sum_{k\geq0} { 2k \choose k } (-1)^k \frac{x^k}{(1-3x)^{k+2}}
     = \sum_{k\geq0} { 2k \choose k } (-1)^k \sum_{n\geq0} { n+1 \choose k+1 } 3^{n-k} x^n
 \end{eqnarray*}
 from which we obtain identity (\ref{id-Grand-Motzkin-edges-03}).

 From series (\ref{series-Grand-Motzkin-edges}), we have (for $\; n \;$ sufficiently large) the identity
 $$ \ell(\GM_n) = 2[x^{n-2}] (1+x)^{-1/2}\left(1-\frac{x}{1/3}\right)^{\!\!-3/2}\, . $$
 Now, by applying the Darboux theorem, we have
 $$ \ell(\GM_n) \sim 2\,\frac{\psi(\xi)}{\xi^{n-2}}\;\frac{(n-2)^{\alpha-1}}{\Gamma(\alpha)} $$
 where $\; \xi = 1/3 \,$, $\; \psi(x) = (1+x)^{-1/2} \;$ and $\; \alpha = 3/2 \,$.
 Since $\; \psi(\xi) = \sqrt{3}/2 \;$ and $\; \Gamma(\alpha) = \Gamma(3/2) = \sqrt{\pi}/2 \,$,
 we can obtain the asymptotic equivalence (\ref{asym-Grand-Motzkin-edges}).
 Moreover, using the asymptotic expansion for the central trinomial coefficients given in (\ref{asymptotics}),
 we have $\; i(\GM_n) \sim \frac{4}{9}\; n \,$.
 Since $\; 4/9 \simeq 0.44 \,$, the Hasse index is asymptotically quasi Boolean.
\end{proof}

Theorems \ref{thm-Motzkin} and \ref{thm-GrandMotzkin} immediately imply
\begin{proposition}
 The class of Motzkin lattices is Hasse-tamed: $\; i(\MM_n) \sim i(\GM_n) \sim \frac{4}{9}\; n \,$.
\end{proposition}

\section{Schr\"oder lattices and Grand Schr\"oder lattices}

\begin{proposition}\label{prop-delta-nabla-Schroder}
 For any Schr\"oder path $\; \gamma \,$, we have
 $\; | \Delta\gamma | = \omega_{H}(\gamma) + \omega_{DU}(\gamma) \;$
 and $\; | \nabla\gamma | = \omega^*_{H}(\gamma) + \omega_{UD}(\gamma) \,$.
 Similarly, for any Grand Schr\"oder path $\; \gamma \,$, we have
 $\; | \Delta\gamma | = \omega_{H}(\gamma) + \omega_{DU}(\gamma) \;$
 and $\; | \nabla\gamma | = \omega_{H}(\gamma) + \omega_{UD}(\gamma) \,$.
\end{proposition}
\begin{proof}
 In a Schr\"oder lattice, a path $\; \gamma \;$ is covered by all paths
 that can be obtained by replacing a horizontal step $\; H \;$ with a peak $\; UD \,$,
 or a valley $\; DU \;$ with a horizontal step $\; H \,$,
 and covers all paths that can be obtained by replacing
 a horizontal step $\; H \;$ (not a level $\; 0 \,$) with a peak $\; UD \;$
 or a peak $\; UD \;$ with a horizontal step $\; H \,$.
 In a Grand Schr\"oder lattice we have an analogous situation.
\end{proof}

\begin{proposition}\label{prop-Schorder}
 The generating series for the class of Schr\"oder paths with respect to semi-length (marked by $\; x $)
 and horizontal steps $\; H \;$ and valleys $\; DU \;$ (marked by $\; q $) is
 \begin{equation}\label{series-Schorder-qx}
  f(q;x) = \frac{1-x-\sqrt{1-2(1+2q)x+(1-2q)^2x^2}}{2qx(1+(1-q)x)} \, .
 \end{equation}
\end{proposition}
\begin{proof}
 Let $\; h(q;x) \;$ be the generating series for the class $\; \SS \;$
 of all Schr\"oder paths starting with a horizontal step
 and let $\; u(q;x) \;$ be the generating series for the class of all Schr\"oder paths starting with an up step.
 Since any non-empty Schr\"oder path $\; \gamma \;$ decomposes uniquely as
 $\; \gamma = H \gamma' \;$ (with $\; \gamma' \in \SS \,$)
 or $\; \gamma = D\gamma'U\gamma'' \;$ (with $\; \gamma', \gamma'' \in \SS \,$),
 we obtain the linear system
 $$
  \begin{cases}
   f(q;x) = 1 + h(q;x) + u(q;x) \\
   h(q;x) = q x f(q;x) \\
   u(q;x) = x f(q;x) ( 1 + h(q;x) + q u(q;x) ) \, .
  \end{cases}
 $$
 By solving for $\; f \,$, we can obtain series (\ref{series-Schorder-qx}).
\end{proof}

The mirror triangle generated by series (\ref{series-Schorder-qx}) is sequence A090981 in \cite{Sloane}.

\begin{theorem}\label{thm-Schroder}
 The edge generating series for Schr\"oder lattices is
 \begin{equation}\label{series-Schroder-edges}
  \ell_\SS(x) = \frac{(1-x)(1-4x+x^2-(1-x)\sqrt{1-6x+x^2})}{2x\sqrt{1-6x+x^2}} \, .
 \end{equation}
 Moreover, the number of edges in $\; \SS_n \;$ can be expressed in one of the forms
 \begin{align}
  & \ell(\SS_n) = d_n - r_n - d_{n-1} + r_{n-1} \label{id-Schroder-edges} \\
  & \ell(\SS_n) = \sum_{k=0}^n \mchoose{2k}{n-k} { 2 k \choose k } \frac{k}{k+1} \label{id-Schroder-edges-02}\, .
 \end{align}
 Finally, we have the asymptotic equivalences
 \begin{equation}\label{asym-Schroder-edges}
  \ell(\SS_n) \sim \frac{(1+\sqrt{2})^{2n}}{\sqrt{\sqrt{2}n\pi}}
  \qquad\text{and}\qquad
  i(\SS_n) = \frac{\ell(\SS_n)}{|\SS_n|} \sim ( 2 - \sqrt{2} ) n \, .
 \end{equation}
 In particular, the Hasse index of Schr\"oder lattices is asymptotically quasi Boolean.
\end{theorem}
\begin{proof}
 Proposition \ref{prop-delta-nabla-Schroder} implies that (\ref{series-Schorder-qx})
 is the $\Delta$-series for Schr\"oder lattices.
 So, by applying Proposition \ref{prop-edge-delta-nabla}, we obtain series (\ref{series-Schroder-edges}).
 It is easy to see that $\; \ell_\SS(x) = (1-x)(d(x)-r(x)) \,$.
 This implies at once identity (\ref{id-Schroder-edges}).
 Moreover, expanding $\; \ell_\SS(x) = (1-x)(d(x)-r(x)) \;$ as follows
 $$
  \ell_\SS(x) = \frac{1}{\sqrt{1-\frac{4x}{(1-x)^2}}} - \frac{1-\sqrt{1-\frac{4x}{(1-x)^2}}}{2\frac{x}{(1-x)^2}}
  = \sum_{k\geq0} { 2k \choose k } \frac{x^2}{(1-x)^{2k}} -
    \sum_{k\geq0} { 2k \choose k }\frac{1}{k+1} \frac{x^2}{(1-x)^{2k}} \, ,
 $$
 it is straightforward to obtain identity (\ref{id-Schroder-edges-02}).
 Finally, using identity (\ref{id-Schroder-edges}) and the asymptotics in (\ref{asymptotics}),
 we obtain equivalences (\ref{asym-Schroder-edges}).
 Since $\; 2 - \sqrt{2} \simeq 0.58 \,$, the Hasse index is asymptotically quasi Boolean.
\end{proof}

\begin{proposition}
 The generating series for the class of Grand Schr\"oder paths with respect to semi-length (marked by $\; x $)
 and horizontal steps $\; H \;$ and valleys $\; DU \;$ (marked by $\; q $) is
 \begin{equation}\label{series-Delannoy-qx}
  g(q;x) = \frac{1}{\sqrt{1-2(1+2q)x+(1-2q)^2x^2}} \, .
 \end{equation}
\end{proposition}
\begin{proof}
 Let $\; \overline{\SS} \;$ be the class of reflected Schr\"oder paths
 (i.e. Grand Schr\"oder paths never going above the $x$-axis),
 and let $\; \overline{f}(q;x) \;$ be the corresponding generating series.
 Since any non-empty reflected Schr\"oder path $\; \gamma \;$ decomposes uniquely as
 $\; \gamma = H \gamma' \;$ (with $\; \gamma' \in \overline{\SS} \,$)
 or $\; \gamma = D\gamma'U\gamma'' \;$ (with $\; \gamma', \gamma'' \in \overline{\SS} \,$),
 we obtain the linear system
 $$
  \begin{cases}
   \overline{f}(q;x) = 1 + \overline{h}(q;x) + \overline{d}(q;x) \\
   \overline{h}(q;x) = q x \overline{f}(q;x) \\
   \overline{d}(q;x) = x ( q + \overline{f}(q;x) - 1 ) \overline{f}(q;x) \, ,
  \end{cases}
 $$
 where $\; \overline{h}(q;x) \;$ and $\; \overline{d}(q;x) \;$
 are the generating series for the reflected Schr\"oder paths
 starting with a horizontal step and with a down step, respectively.
 By solving for $\; \overline{f}(q;x) \,$, we obtain
 $$ \overline{f}(q;x) = \frac{1+x-2qx-\sqrt{(1+x-2qx)^2-4x}}{2x} \, . $$

 Since any non-empty Grand Schr\"oder path $\; \gamma \;$ decomposes uniquely as
 $\; \gamma = H\gamma' \;$ (with $\; \gamma' \in \GS \,$),
 or as $\; \gamma = U\gamma'D\gamma'' \;$ (with $\; \gamma' \in \SS \;$ and $\; \gamma'' \in \GS \,$),
 or as $\; \gamma = D\gamma'U\gamma'' \;$ (with $\; \gamma' \in \overline{\SS} \;$ and $\; \gamma'' \in \GS \,$),
 we have the system
 $$
  \begin{cases}
   g(q;x) = 1 + h(q;x) + u(q;x) + d(q;x) \\
   h(q;x) = q x g(q;x) \\
   u(q;x) = x f(q;x) ( 1 + h(q;x) + q u(q;x) + d(q;x) ) \\
   d(q;x) = x ( q + \overline{f}(q;x) - 1 ) g(q;x) \, ,
  \end{cases}
 $$
 where $\; h(q;x) \,$, $\; u(q;x) \;$ and $\; d(q;x) \;$
 are the generating series for the Grand Schr\"oder paths
 starting with a horizontal step, an up step and with a down step, respectively.
 By solving for $\; g(q;x) \,$, it is straightforward to obtain series (\ref{series-Delannoy-qx}).
\end{proof}

\begin{theorem}\label{thm-GrandSchroder}
 The edge generating series for Grand Schr\"oder lattices is
 \begin{equation}\label{series-Delannoy-edges}
  \ell_\GS(x) = \frac{2(x-x^2)}{(1-6x+x^2)^{3/2}} \, .
 \end{equation}
 Moreover, we have the identities
 \begin{align}
  & \ell(\GS_n) = 2 \sum_{k=0}^n { n+k \choose 2k } { 2k \choose k } (n-k) \label{id-Delannoy-edges-1} \\
  & \ell(\GS_n) = \sum_{k=0}^{\lfloor n/2 \rfloor} { n \choose k } { n-k \choose k } \frac{(n-2k)(n+k+2)}{k+1}\; 2^{k+1} 3^{n-2k-2} \, . \label{id-Delannoy-edges-2}
 \end{align}
 and the asymptotic equivalences
 \begin{equation}\label{asym-Delannoy-edges}
  \ell(\GS_n) \sim \sqrt{\frac{n}{2\sqrt{2}\pi}}\; (1+\sqrt{2})^{2n}
  \qquad\text{and}\qquad
  i(\GS_n) = \frac{\ell(\GS_n)}{|\GS_n|} \sim ( 2 - \sqrt{2} ) n \, .
 \end{equation}
 In particular, the Hasse index of Grand Schr\"oder lattices is asymptotically quasi Boolean.
\end{theorem}
\begin{proof}
 By Proposition \ref{prop-delta-nabla-Schroder},
 the $\Delta$-series for Grand Schr\"oder lattices is series (\ref{series-Delannoy-qx}).
 So, by Proposition \ref{prop-edge-delta-nabla}, we obtain series (\ref{series-Delannoy-edges}).
 By expanding this series as follows
 $$
  \ell_\GS(x) = \frac{2x(1-x)}{((1-x)^2-4x)^{3/2}}
  = \frac{2x}{(1-x)^3}\;\frac{1}{\left(1-\frac{4x}{(1-x)^2}\right)^{3/2}}
  = 2 \sum_{k\geq0} { 2k \choose k } (2k+1) \frac{x^{k+1}}{(1-x)^{2k+2}}
 $$
 we can obtain identity (\ref{id-Delannoy-edges-1}).
 Formula (\ref{id-Delannoy-edges-2}) can be obtained in a similar way
 (just consider the identity $\; 1-6x+x^2 = (1-3x)^2-8x^2 \,$).
 Finally, using the Darboux theorem,
 we can obtain the first equivalence in (\ref{asym-Delannoy-edges}).
 Then, using the asymptotics in (\ref{asymptotics}),
 we can also obtain the second equivalence in (\ref{asym-Delannoy-edges}).
\end{proof}

Theorems \ref{thm-Schroder} and \ref{thm-GrandSchroder} immediately imply
\begin{proposition}
 The class of Schr\"oder lattices is Hasse-tamed: $\; i(\SS_n) \sim i(\GS_n) \sim ( 2 - \sqrt{2} ) n \,$.
\end{proposition}

\section{Fibonacci and Grand Fibonacci posets}\label{sec-Fibo}

All classes of paths considered in the previous sections have the property of being Hasse-tamed.
Here, we will consider a simple class of paths for which such a property is not true.

A \emph{Fibonacci path} is a Motzkin path confined in the horizontal strip $\; [0,1] \;$
with horizontal steps only on the $x$-axis.
A \emph{Grand Fibonacci path} is a Grand Motzkin path confined in the horizontal strip $\; [-1,1] \;$
with horizontal steps only on the $x$-axis.
Let $\; \FF \;$ be the class of \emph{Fibonacci posets}
(i.e. \emph{Fibonacci semilattices} \cite{FontanesiMunariniZagaglia}),
and let $\; \GF \;$ be the class of \emph{Grand Fibonacci posets}.

Let $\; F_n \;$ and $\; L_n \;$ be the \emph{Fibonacci} and \emph{Lucas numbers}, respectively.
By Binet formulas, we have the identities
$\; F_n = (\ff^n-\widehat{\ff})/\sqrt{5} \;$ and $\; L_n = \ff^n+\widehat{\ff} \,$,
and the asymptotics $\; F_n \sim \ff^n/\sqrt{5} \;$ and $\; L_n \sim \ff^n \,$,
where $\; \ff = (1+\sqrt{5})/2 \;$ and $\; \widehat{\ff} = (1-\sqrt{5})/2 \,$.
\begin{theorem}\label{thm-Fibo}
 The edge generating series for the Fibonacci posets is
 \begin{equation}\label{series-Fibonacci-edges}
  \ell_\FF(x) = \frac{x^2}{(1-x-x^2)^2} \, .
 \end{equation}
 Moreover, we have the explicit formula
 \begin{equation}\label{id-Fibonacci-edges}
  \ell(\FF_n) = \frac{nL_n-F_n}{5}
 \end{equation}
 and the asymptotic equivalences
 \begin{equation}\label{asymptotic-Fibonacci-edges}
  \ell(\FF_n) \sim \frac{n}{5}\;\ff^n
  \qquad\text{and}\qquad
  i(\FF_n) \sim \frac{n}{\sqrt{5}\,\ff}\, .
 \end{equation}
 In particular, the Fibonacci posets are not asymptotically quasi boolean.
\end{theorem}
\begin{proof}
 In a Fibonacci poset, a path $\; \gamma \;$ is covered by all paths that can be obtained from $\; \gamma \;$
 by replacing a horizontal double step $\; HH \;$ (necessarily on the $x$-axis) with a peak $\; UD \,$.
 So, we have $\; |\Delta\gamma| = \omega_{HH}(\gamma) \,$.
 Let $\; f(q;x) \;$ be the generating series for Fibonacci paths
 with respect to horizontal double steps and length.
 Similarly, let $\; h(q;x) \;$ and $\; u(q;x) \;$ be the generating series for Fibonacci paths
 starting with a horizontal step or with an up step, respectively.
 Since any non-empty Fibonacci path $\; \gamma \;$ decomposes uniquely as
 $\; \gamma = H\gamma' \;$ or as $\; \gamma = UD\gamma' \,$,
 where $\; \gamma' \;$ is an arbitrary Fibonacci path in both cases,
 it is easy to obtain the identities:
 $\; f(q;x) = 1 + h(q;x) + u(q;x) \,$, $\; h(q;x) = x(1+qh(q;x)+u(q;x)) \,$, $\; u(q;x) = x^2f(q;x) \,$.
 The solution of this linear system is straightforward and yields the series
 $$ f(q;x) = \frac{1+(1-q)x}{1-qx-x^2-(1-q)x^3} \, . $$
 Now, by applying Proposition \ref{prop-edge-delta-nabla}, we can obtain series (\ref{series-Fibonacci-edges}).
 Expanding this series we obtain identity (\ref{id-Fibonacci-edges}).
 From this identity and the identity $\; i(\FF_n) = \frac{nL_n-F_n}{5F_{n+1}} \,$,
 we obtain at once asymptotics (\ref{asymptotic-Fibonacci-edges}).
 Finally, since $\; 1/(\sqrt{5}\ff) \simeq 0.276 \,$, the Fibonacci posets are not asymptotically quasi boolean.
\end{proof}

\begin{theorem}\label{thm-Grand-Fibo}
 The edge generating series for the Grand Fibonacci posets is
 \begin{equation}\label{series-GrandFibonacci-edges}
  \ell_\GF(x) = \frac{2x^2}{(1-x-2x^2)^2}\, .
 \end{equation}
 Moreover, we have the explicit formula
 \begin{equation}\label{id-GrandFibonacci-edges}
  \ell(\GF_n) = \frac{(3n-1)2^{n+1}+(-1)^n2(3n+1)}{27}
 \end{equation}
 and the asymptotic equivalences
 \begin{equation}\label{asymptotic-GrandFibonacci-edges}
  \ell(\GF_n) \sim \frac{n}{9}\;2^{n+1}
  \qquad\text{and}\qquad
  i(\GF_n) \sim \frac{n}{3}\, .
 \end{equation}
 In particular, the Grand Fibonacci posets are not asymptotically quasi boolean.
\end{theorem}
\begin{proof}
 In a Grand Fibonacci poset, a path $\; \gamma \;$ is covered by all paths that can be obtained from $\; \gamma \;$
 by replacing a horizontal double step $\; HH \;$ (necessarily on the $x$-axis) with a peak $\; UD \,$,
 or by replacing a valley $\; DU \;$ touching the line $\; y = - 1 \;$
 with a horizontal double step $\; HH \;$ on the $x$-axis.
 Then, let $\; F(q;x) \;$ be the generating series for Grand Fibonacci paths
 with respect to horizontal double steps and valleys touching the line $\; y = - 1 \;$ (marked by $\; q \,$)
 and length (marked by $\; x \,$).
 Similarly, let $\; H(q;x) \,$, $\; U(q;x) \;$ and $\; D(q;x) \;$ be the generating series for these paths
 starting with a horizontal step, with an up step or with a down step, respectively.
 Since any non-empty Grand Fibonacci path $\; \gamma \;$ decomposes uniquely as
 $\; \gamma = H\gamma' \,$, as $\; \gamma = UD\gamma' \,$, or as $\; \gamma = DU\gamma' \,$,
 where $\; \gamma' \;$ is an arbitrary Grand Fibonacci path in each case,
 it is easy to obtain the identities:
 $\; F(q;x) = 1 + H(q;x) + U(q;x) + D(q;x) \,$, $\; H(q;x) = x(1+qH(q;x)+U(q;x)+D(q;x)) \,$,
 $\; U(q;x) = x^2F(q;x) \,$, $\; D(q;x) = qx^2F(q;x) \,$.
 The solution of this linear system is straightforward and yields the series
 $$ F(q;x) = \frac{1+(1-q)x}{1-qx-(1+q)x^2-(1-q^2)x^3} \, . $$
 Now, by applying Proposition \ref{prop-edge-delta-nabla}, we can obtain series (\ref{series-GrandFibonacci-edges}).
 Decomposing this series as sum of partial fractions, it is easy to obtain identity (\ref{id-GrandFibonacci-edges}).
 From this identity, we can obtain the first relation in (\ref{asymptotic-GrandFibonacci-edges}).
 Finally, notice that the generating series for the vertices of the Grand Fibonacci posets
 is $\; g(x) = F(1;x) = 1/(1-x-2x^2) \;$
 and that consequently the number of vertices is $\; g_n = (2^{n+1}+(-1)^n)/3 \sim 2^{n+1}/3 \,$.
 So, we also have the second relation in (\ref{asymptotic-GrandFibonacci-edges}).
\end{proof}

Theorems \ref{thm-Fibo} and \ref{thm-Grand-Fibo} immediately imply
\begin{proposition}
 The class of Fibonacci lattices is not Hasse-tamed.
\end{proposition}

\section{Young lattices}

A \emph{partition} of a non-negative integer $\; n \;$ is a sequence $\; \lambda = (\ll_1,\ll_2,\ldots,\ll_k) \;$
of positive integers such that $\; \ll_1 \geq \ll_2 \geq \cdots \geq \ll_k > 0 \;$
and $\; |\lambda | = \ll_1 + \ll_2 + \cdots + \ll_k = n \,$.
The $\; \ll_i$'s are the \emph{parts} of $\; \ll \,$.
The (\emph{Ferrers}) \emph{diagram} of $\; \ll \;$ is a left-justified array of squares (or dots)
with exactly $\; \ll_i \;$ squares in the $\; i$-th row.
Partitions can be ordered by magnitude of parts \cite{Aigner}:
if $\; \alpha = (a_1,\ldots,a_h) \;$ and $\; \beta = (b_1,\ldots,b_k) \,$, then $\; \alpha \leq \beta \;$
whenever $\; h \leq k \;$ and $\; a_i \leq b_i \;$ for every $\; i = 1, 2, \ldots, h \,$.
If $\; \alpha \leq \beta \;$ the diagram of $\; \alpha \;$ is contained in the diagram of $\; \beta \,$.
The resulting poset is an infinite distributive lattice, called \emph{Young lattice}.
The \emph{Young lattice} $\; \YY_\lambda \;$ generated by a partition $\; \lambda \;$
is the principal ideal generated by $\; \lambda \;$ in $\; \YY \,$,
i.e. the set of all integer partitions $\; \alpha \;$ such that $\; \alpha \leq \lambda \,$.
Any Young lattice $\; \YY_\lambda \;$ can be considered as a lattice of paths.
Indeed, it is isomorphic to the lattice of all paths
made of horizontal steps $\; X = (1,0) \;$ and vertical steps $\; Y = (0,1) \,$,
confined in the region defined by the diagram $\; \Phi_\ll \;$ of $\; \ll \,$.
Moreover, as we noticed in \cite{FerrariMunarini},
Dyck lattices are isomorphic to the dual of the Young lattices associated with
the staircase partitions $\; (n,n-1,\ldots,2,1) \,$.

Here we will give a general formula for the number of edges in any Young lattice $\; \YY_\lambda \,$.
To do that we need the following definitions.
We say that a cell of a Young diagram is a \emph{corner cell}
whenever the diagram has no cells below and no cells on the right of such a cell.
The \emph{boundary} of the diagram $\; \Phi_\ll \;$ is the set $\; \partial\Phi_\ll \;$ of all its corner cells.

\begin{figure}
\begin{center}
 \setlength{\unitlength}{4mm}
 \begin{picture}(28,10)
 \put(0,0){
 \begin{picture}(12,10)
  \put(0,6.3){\colorbox{yellow}{\makebox(4.45,2.41)[tr]{}}}
  \put(4.5,6.5){\makebox(0,0){$\square$}}
  \put(0,0.3){\colorbox{green}{\makebox(0.45,5.41)[tr]{}}}
  \put(1,1.3){\colorbox{green}{\makebox(0.45,4.41)[tr]{}}}
  \put(2,2.3){\colorbox{green}{\makebox(1.45,3.41)[tr]{}}}
  \put(5,8.3){\colorbox{green}{\makebox(6.45,0.41)[tr]{}}}
  \put(5,7.3){\colorbox{green}{\makebox(4.45,0.41)[tr]{}}}
  \put(0,0){\line(0,1){9}}
  \put(1,0){\line(0,1){9}}
  \put(2,1){\line(0,1){8}}
  \put(3,2){\line(0,1){7}}
  \put(4,2){\line(0,1){7}}
  \put(5,2){\line(0,1){7}}
  \put(6,2){\line(0,1){7}}
  \put(7,5){\line(0,1){4}}
  \put(8,5){\line(0,1){4}}
  \put(9,6){\line(0,1){3}}
  \put(10,6){\line(0,1){3}}
  \put(11,8){\line(0,1){1}}
  \put(12,8){\line(0,1){1}}
  \put(0,0){\line(1,0){1}}
  \put(0,1){\line(1,0){2}}
  \put(0,2){\line(1,0){6}}
  \put(0,3){\line(1,0){6}}
  \put(0,4){\line(1,0){6}}
  \put(0,5){\line(1,0){8}}
  \put(0,6){\line(1,0){10}}
  \put(0,7){\line(1,0){10}}
  \put(0,8){\line(1,0){12}}
  \put(0,9){\line(1,0){12}}
  \put(-0.5,6.5){\makebox(0,0){$i$}}
  \put(4.5,9.5){\makebox(0,0){$j$}}
  \put(6,-0.5){\makebox(0,0){(a)}}
 \end{picture}}
 \put(16,0){
 \begin{picture}(12,10)
  \put(0,6.3){\colorbox{yellow}{\makebox(4.45,2.41)[tr]{}}}
  \put(4.5,6.5){\makebox(0,0){$\square$}}
  \put(0,0.3){\colorbox{green}{\makebox(0.45,5.41)[tr]{}}}
  \put(1,1.3){\colorbox{green}{\makebox(0.45,4.41)[tr]{}}}
  \put(2,2.3){\colorbox{green}{\makebox(1.45,3.41)[tr]{}}}
  \put(5,8.3){\colorbox{green}{\makebox(6.45,0.41)[tr]{}}}
  \put(5,7.3){\colorbox{green}{\makebox(4.45,0.41)[tr]{}}}
  \put(0,0){\line(0,1){9}}
  \put(1,0){\line(0,1){9}}
  \put(2,1){\line(0,1){8}}
  \put(3,2){\line(0,1){7}}
  \put(4,2){\line(0,1){7}}
  \put(5,6){\line(0,1){3}}
  \put(6,7){\line(0,1){2}}
  \put(7,7){\line(0,1){2}}
  \put(8,7){\line(0,1){2}}
  \put(9,7){\line(0,1){2}}
  \put(10,7){\line(0,1){2}}
  \put(11,8){\line(0,1){1}}
  \put(12,8){\line(0,1){1}}
  \put(0,0){\line(1,0){1}}
  \put(0,1){\line(1,0){2}}
  \put(0,2){\line(1,0){4}}
  \put(0,3){\line(1,0){4}}
  \put(0,4){\line(1,0){4}}
  \put(0,5){\line(1,0){4}}
  \put(0,6){\line(1,0){5}}
  \put(0,7){\line(1,0){10}}
  \put(0,8){\line(1,0){12}}
  \put(0,9){\line(1,0){12}}
  \put(-0.5,6.5){\makebox(0,0){$i$}}
  \put(4.5,9.5){\makebox(0,0){$j$}}
  \put(8.5,9.7){\makebox(0,0){$\ll_{ij}^\nearrow$}}
  \put(-1,3){\makebox(0,0){$\ll_{ij}^\swarrow$}}
  \put(6,-0.5){\makebox(0,0){(b)}}
 \end{picture}}
 \end{picture}
\end{center}
 \caption{(a) Diagram of the partition $\; \ll = (12,10,10,8,6,6,6,2,1) \,$,
  with the cell $\; (i,j) = (3,5) \;$ marked by a square.
  (b) Diagram of the partition $\; \ll_{ij}^{\searrow} = (12,10,5,4,4,4,4,2,1) \,$.
  In both pictures, the diagrams of the partitions $\; \ll_{ij}^\nearrow = (7,5) \;$
  and $\; \ll_{ij}^\swarrow = (4,4,4,4,2,1) \;$
  are in dark color.}
 \label{FigFerrersDiagram}
\end{figure}
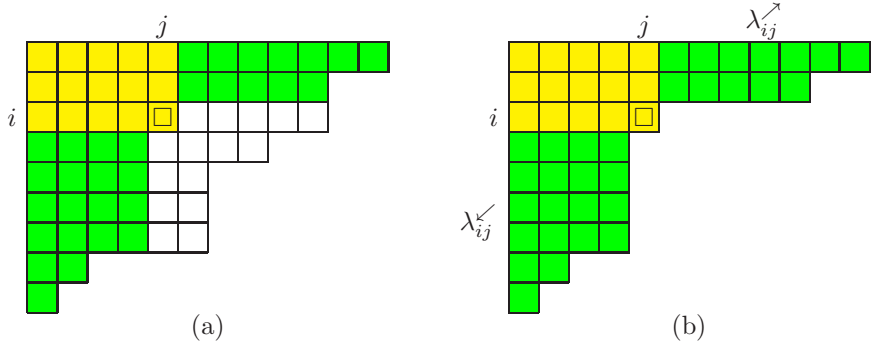

In a Ferrers diagram $\; \Phi_\ll \,$,
we denote by $\; (i,j) \;$ the cell of $\; \Phi_\ll \;$
lying at the cross of the $i$-th row with the $j$-th column.
For instance, $\; (1,1) \;$ is the cell of $\; \Phi_\ll \;$ in the top left corner.
Given a cell $\; (i,j) \;$ of $\; \Phi_\ll \,$,
we have the partition $\; \ll_{ij}^\nwarrow \;$ obtained by intersecting
of the first $\; i \;$ rows and the first $\; j \;$ columns of $\; \Phi_\ll \;$
(the light colored cells in Figure \ref{FigFerrersDiagram}(a)),
the partition $\; \ll_{ij}^\searrow \;$ obtained by deleting all cells $\; (i',j') \ne (i,j) \;$
such that $\; i' \geq i \;$ and $\; j' \geq j \;$
(the white cells in Figure \ref{FigFerrersDiagram}(a)),
the partition $\; \ll_{ij}^\nearrow \;$ obtained by intersecting
the first $\; i-1 \;$ rows and the last $\; \ll_1-j \;$ columns of $\; \Phi_\ll \;$
(the dark cells above $\; (i,j) \;$ in Figure \ref{FigFerrersDiagram}(a)),
and finally the partition $\; \ll_{ij}^\swarrow \;$ obtained by intersecting
the last $\; k-i \;$ rows and the first $\; j-1 \;$ columns of $\; \Phi_\ll \;$
(the dark cells below $\; (i,j) \;$ in Figure \ref{FigFerrersDiagram}(a)).
\begin{theorem}\label{thm-Young}
 The number of edges in the Hasse diagram of the Young lattice $\; \YY_\ll \;$ is
 \begin{equation}\label{edge-Young-lattice}
  \ell(\YY_\ll) =  \sum_{(i,j)\in\Phi_\ll} |\YY_{\ll_{ij}^\swarrow}|\; |\YY_{\ll_{ij}^\nearrow}| \, .
 \end{equation}
\end{theorem}
\begin{proof}
 Let $\; \mu \in \YY_\ll \,$.
 The partitions covered by $\; \mu \;$ are exactly the partitions
 obtained by removing a cell from the boundary of $\; \mu \,$.
 Hence $\; |\nabla\mu| = |\partial\Phi_\mu| \;$ and
 $$
  \ell(\YY_\ll) = \sum_{\mu\in\YY_\ll} |\nabla \mu|
  = \sum_{\mu\in\YY_\ll} |\partial\Phi_\mu|
  = \sum_{(i,j)\in\Phi_\ll} C_{ij}
 $$
 where $\; C_{ij} \;$ is the number of all partitions $\; \mu \in \YY_\ll \;$
 such that $\; (i,j) \in \partial\Phi_\mu \,$.

 Since $\; \ll_{ij}^\nwarrow \;$ and $\; \ll_{ij}^\searrow \;$ are respectively
 the minimal and the maximal partition in $\; \YY_\ll \;$ admitting $\; (i,j) \;$ as a corner cell,
 a partition $\; \mu \;$ has $\; (i,j) \;$ as a corner cell
 if and only if $\; \mu\in[\ll_{ij}^\nwarrow,\ll_{ij}^\searrow] \,$.
 This implies that the partition $\; \mu \;$ is equivalent to a pair $\; (\mu_1,\mu_2) \;$ of partitions,
 where $\; \Phi_{\mu_1} \subseteq \Phi_{\ll_{ij}^\swarrow} \;$
 and $\; \Phi_{\mu_2} \subseteq \Phi_{\ll_{ij}^\nearrow} \,$.
 So $\; C_{ij} = |\YY_{\ll_{ij}^\swarrow}|\; |\YY_{\ll_{ij}^\nearrow}| \,$,
 and this yields identity (\ref{edge-Young-lattice}).
\end{proof}

To illustrate Theorem \ref{thm-Young},
we consider the particularly interesting case of the Young lattices $\; L(m,n) = \YY_\ll \;$ (for $\; m, n \geq 1 \,$)
obtained for the partition $\; \ll = (n,\ldots,n) \;$ with $\; m \;$ parts
(see, for instance, \cite{Stanley0,Proctor}).
Since $\; L(n,n) = \GD_n \,$, the next theorem generalizes Theorem \ref{thm-GrandDyck}.
\begin{theorem}
 The number of edges in the Hasse diagram of the Young lattice $\; L(m,n) \;$ is
 \begin{equation}\label{id-Lmn-edges}
  \ell(L(m,n)) =  { m+n-1 \choose n } n = \frac{(m+n-1)!}{(m-1)!(n-1)!} \, .
 \end{equation}
\end{theorem}
\begin{proof}
 For every cell $\; (i,j) \in \Phi_\ll \,$, we have
 $\; \ll_{ij}^\nearrow = (n-j)^{i-1} \;$ and $\; \ll_{ij}^\swarrow = (j-1)^{m-i} \,$.
 Hence, by formula (\ref{edge-Young-lattice}), we have
 $$ \ell(L(m,n)) = \sum_{i=1}^m \sum_{j=1}^n |L(n-j,i-1)|\, |L(j-1,m-i)|\, . $$
 Since $\; |L(m,n)| = { m + n \choose n } \,$, we have
 $$ \ell(L(m,n)) = \sum_{i=1}^m \sum_{j=1}^n { m - i + j - 1 \choose j - 1 } { n + i - j - 1 \choose i - 1 } \, . $$
 The generating series of these numbers is
 \begin{eqnarray*}
  L(x,y) & = & \sum_{m,n\geq1} \ell(L(m,n))\; x^m y^n \\
         & = & \sum_{i,j\geq1}
               \left[\sum_{m\geq1} { m-i+j-1 \choose j-1 } x^m\right]
               \left[\sum_{n\geq1} { n+i-j-1 \choose i-1 } y^n\right] \\
         & = & \sum_{i,j\geq0}
               \left[\sum_{m\geq0} { m-i+j \choose j } x^{m+1}\right]
               \left[\sum_{n\geq0} { n+i-j \choose i } y^{n+1}\right] \\
         & = & \sum_{i,j\geq0} \frac{x^{i+1}}{(1-x)^{j+1}}\; \frac{y^{j+1}}{(1-y)^{i+1}}
           =   \sum_{i\geq0} \frac{x^{+1i}}{(1-y)^{i+1}}\; \sum_{j\geq0} \frac{y^{j+1}}{(1-x)^{j+1}} \\
         & = & \frac{xy}{(1-x-y)^2}
           =   xy\;\frac{\partial}{\partial y}\frac{1}{1-x-y}
           =   \sum_{m,n\geq0} { m+n-1 \choose n } n\; x^m y^n \, .
 \end{eqnarray*}
 Hence, we have identity (\ref{id-Lmn-edges}).
\end{proof}




\end{document}